\documentclass[12pt,reqno]{amsart}  

\usepackage{amsmath,amsthm,amssymb,amscd,mathdots,mathtools,enumerate,float,shuffle,stmaryrd}

\usepackage{hyperref}
\usepackage{tikz,tikz-cd}
\allowdisplaybreaks
\usetikzlibrary{decorations.pathmorphing,shapes,calc}

\newtheorem{theorem}{Theorem}[section]
\newtheorem*{theorem*}{Theorem}
\newtheorem{proposition}[theorem]{Proposition}

\newtheorem{lemma}[theorem]{Lemma}
\newtheorem{corollary}[theorem]{Corollary}

\numberwithin{equation}{section}

\theoremstyle{definition}
\newtheorem{definition}[theorem]{Definition}
\newtheorem{example}[theorem]{Example}

\mathtoolsset{showonlyrefs} % just show used refs

\newcommand{\h}{\mathfrak H}

\newcommand{\Ha}{\mathbb{H}}
\newcommand{\Q}{\mathbb{Q}}
\newcommand{\Z}{\mathbb{Z}}
\newcommand{\R}{\mathbb{R}}

\newcommand{\C}{\mathbb{C}}

\newcommand{\hz}{\h^{2}}

\newcommand{\one}{{\bf 1}}
\newcommand{\qsh}{\ast_\diamond}

\newcommand{\QL}{\Q\langle L \rangle}

\newcommand{\mz}{\mathcal Z}

\definecolor{mycolor}{RGB}{194, 8, 88}
\newcommand{\todo}[1]{\message{LaTeX Warning: You did not finish your work :-( on input line \the\inputlineno} {\color{mycolor} {\big[\,}{\bf Todo:} #1\,\big]}}

\newcommand{\sh}{\shuffle}
\newcommand{\kk}{{\bf k}}

\makeatletter
\newcommand{\cop}{\DOTSB\cop@\slimits@}
\newcommand{\cop@}{\mathop{\bigstar}}
\newcommand{\cop@@}[2]{%
  \vphantom{\sum}%
  \ifx#1\displaystyle\big#2\else#2\fi
}
\makeatother

\title{Stuffle regularized multiple Eisenstein series revisited}

\author{Henrik Bachmann}
\address{Graduate School of Mathematics,  Nagoya University, Nagoya, Japan.}
\email{henrik.bachmann@math.nagoya-u.ac.jp}


\begin{document}
\date{\today}
\maketitle

\begin{abstract} 
Multiple Eisenstein series are holomorphic functions in the complex upper-half plane, which can be seen as a crossbreed between multiple zeta values and classical Eisenstein series. They were originally defined by Gangl-Kaneko-Zagier in 2006, and since then, many variants and regularizations of them have been studied. They give a natural bridge between the world of modular forms and multiple zeta values. 
In this note, we give a new algebraic interpretation of stuffle regularized multiple Eisenstein series based on the Hopf algebra structure of the harmonic algebra introduced by Hoffman. 
\end{abstract}

\section{Introduction}

In this note, we will give an overview of the regularization of multiple Eisenstein series and present a new Hopf algebraic approach for the construction of stuffle regularized multiple Eisenstein series. Multiple Eisenstein series were introduced in the depth two case by Gangl-Kaneko-Zagier in \cite{GKZ}, and for higher depth, they were studied in \cite{B0},\cite{B1},\cite{B2} and \cite{BT}. Similar to multiple zeta values, there exist two regularizations, the shuffle and stuffle regularizations, which were introduced in \cite{BT} and \cite{B1}, respectively. We will focus on the stuffle regularisation in this work and give a new algebraic interpretation for them.

As the main building block for all the objects appearing in this note, we define for an index $\kk=(k_1,\dots,k_r) \in \Z_{\geq 1}^r$ of depth $r\geq 1$, $x \in \C \backslash \Z_{<0}$ and $N\geq 1$ the \emph{truncated multiple Hurwitz zeta function} by
\begin{align}\label{eq:defmultiplhurwitz}
       \zeta_N(\kk; x) :=  \zeta_N(k_1,\dots,k_r; x) := \sum_{N > n_1 > \dots > n_r >0} \frac{1}{(x+n_1)^{k_1} \dots (x + n_r)^{k_r}}\,.
\end{align}
The index $\kk$ is also allowed to be empty (i.e. $r=0$), in which case we set $\zeta_N(\emptyset; x) :=1$. We write  $\zeta_N(\kk) =  \zeta_N(\kk; 0)$ for the \emph{truncated multiple zeta values}. In the case $k_1\geq 2$ or $\kk=\emptyset$ the index $\kk$ is called admissible and we can take the limit $N\rightarrow \infty$ to obtain the \emph{multiple zeta values} $\zeta(\kk)  = \lim_{N\rightarrow \infty} \zeta_N(\kk)$, which specialize to the Riemann zeta values $\zeta(k)$ in the depth $r=1$ case. By $\mz$ we denote the $\Q$-vector space spanned by all multiple zeta values. 

The Riemann zeta values also appear as the constant term in the Fourier expansion of the Eisenstein series $\mathbb
{G}(k)$, which can also be constructed out of \eqref{eq:defmultiplhurwitz} in the following way: For $k,N\geq 1$ and  $x \in \C \backslash \Z$ we define
\begin{align}\label{eq:truncatedmonotangent}
        \Psi_N(k;x) := \zeta_N(k; x) +\frac{1}{x} + (-1)^k \zeta_{N}(k;-x) = \sum_{N > n > -N} \frac{1}{(x+n)^k}\,.
\end{align}
 Let $\tau \in \Ha = \{ z \in \C \mid \Im(z) > 0\}$ be an element in the upper-half plane and assume $k\geq 2$. In this case, the limit $\Psi(k;\tau) := \lim_{N \rightarrow \infty} \Psi_N(k;\tau)$ exists and we get by the Lipschitz formula
 \begin{align}\label{eq:defmonotangent}
    \Psi(k;\tau) = \sum_{n \in \Z} \frac{1}{(\tau+n)^k} = \frac{(-2\pi i)^{k}}{(k-1)!} \sum_{d>0} d^{k-1} q^d  \,,
\end{align}
 where $q=e^{2 \pi i \tau}$. In particular, we can also replace $\tau$ by $m\tau$ and take the sum over all $m\geq 1$ to obtain 
 \begin{align*}
     \sum_{m\geq 1} \Psi(k;m\tau) = \frac{(-2\pi i)^k}{(k-1)!}\sum_{d>0} \frac{d^{k-1} q^d}{1-q^d} \,,
\end{align*}
 which is, up to the constant term $\zeta(k)$, exactly the Fourier expansion of the Eisenstein series of weight $k$. This shows that the truncated multiple Hurwitz zeta function \eqref{eq:defmultiplhurwitz} can be used to construct both multiple zeta values as well as the classical Eisenstein series. The purpose of this note is to show that they can also be used to construct multiple Eisenstein series. This will then be used to show that the regularization of the multiple Hurwitz zeta function gives a, in some sense, natural way of defining stuffle regularized multiple Eisenstein series. This construction is based on the Hopf algebra structure of quasi-shuffle algebras, which we will recall in Section \ref{sec:algsetup} after recalling the basic calculation of multiple Eisenstein series in Section \ref{sec:mes}. In Section \ref{sec:stmes}, we then show how (truncated) multiple Eisenstein series can be constructed out of the truncated multiple Hurwitz zeta function \eqref{eq:defmultiplhurwitz} using the convolution product and standard regularization techniques in the harmonic algebra. In the end, we mention some new results on the comparison between the stuffle and shuffle regularized multiple Eisenstein series, which will be discussed in more detail in the master thesis of Turan in \cite{T}. Finally, we mention that the construction we present in Section \ref{sec:stmes} is similar to the one used in \cite{BB} to construct combinatorial multiple Eisenstein series.  \\

{\bf Acknowledgement:} This project was partially supported by JSPS KAKENHI Grants 19K14499 and 21K13771. 

% \begin{figure}[!ht]
% 	\centering
% 	\input{overview}
% \end{figure}

\section{Multiple Eisenstein series}\label{sec:mes}
In this section, we recall basic facts on multiple zeta values, multiple Eisenstein series, and the calculation of their Fourier expansion. Details can be found in \cite{B0}, \cite{B1},\cite{B2},\cite{B3}, and \cite{BT}. For $ k_1,\dots,k_r \geq 2$ and $\tau \in \Ha$ the \emph{multiple Eisenstein series} are defined\footnote{In the case $k_1=2$ we need to use Eisenstein summation as done in Section \ref{sec:stmes}} by
\begin{align}\label{eq:defmes}
\mathbb{G}(k_1,\dots,k_r;\tau) := \sum_{\substack{\lambda_1 \succ \dots \succ \lambda_r \succ 0\\ \lambda_i \in \Z \tau + \Z}} \frac{1}{\lambda_1^{k_1} \dots \lambda_r^{k_r}}   \,,
\end{align}
where the order $\succ$ on the lattice $\Z \tau + \Z$ is defined by $m_1 \tau + n_1 \succ m_2 \tau + n_2$ iff $m_1 > m_2$ or $m_1 = m_2 \wedge n_1 > n_2$. Since $\mathbb{G}(k_1,\dots,k_r;\tau + 1) = \mathbb{G}(k_1,\dots,k_r;\tau)$ the multiple Eisenstein series possess a Fourier expansion, i.e., an expansion in $q=e^{2\pi i \tau}$, which was calculated in \cite{GKZ} for the $r=2$ case and for arbitrary depth by the first author (\cite{B2}). In depth one, we have for $k\geq 2$
\begin{align*} \mathbb{G}(k;\tau) = \sum_{\substack{\lambda \in \Z \tau + \Z\\  \lambda \succ 0}} \frac{1}{\lambda^k }  = \sum_{ \substack{ m > 0 \\ \vee \, (   m=0 \wedge n>0) }} \frac{1}{(m\tau +n)^k }  =  \zeta(k)+  \sum_{m>0} \underbrace{\sum_{n\in \Z}\frac{1}{(m\tau +n)^k}}_{{\large =: \Psi(k;m\tau) } }\,.
\end{align*}
For even $k\geq 4$, these are just the classical Eisenstein series, which are modular forms for the full modular group. Here the $\Psi(k;\tau)$ are exactly the function we saw in \eqref{eq:defmonotangent}, and we refer to them as the \emph{monotangent function} (\cite{Bo}). By the Lipschitz formula \eqref{eq:defmonotangent} we obtain
\begin{align*}
\mathbb{G}(k;\tau)  &= \zeta(k) + \sum_{m>0}  \Psi_k(m\tau) =  \zeta(k) + \frac{(-2\pi i)^{k}}{(k-1)!}\sum_{\substack{m>0\\ d >0}} d^{k-1} q^{m d} =: \zeta(k) + (-2\pi i)^k g(k) \,.
\end{align*} 
Here the $g(k)$ are the generating series of the divisor-sums, and for higher depths, multiple versions of these $q$-series appear, which are defined for $k_1,\dots k_r \geq 1$ by
\begin{align}\label{eq:defmonog}
g(k_1,\dots,k_r;\tau)=g(k_1,\dots,k_r)= \sum_{\substack{m_1 > \dots > m_r > 0\\ n_1, \dots , n_r > 0}} \frac{n_1^{k_1-1}}{(k_1-1)!} \dots \frac{n_r^{k_r-1}}{(k_r-1)!}  q^{m_1 n_1 + \dots + m_r n_r } \,.
\end{align}
These $q$-series were studied in detail in \cite{B2}, \cite{BK} and they can be seen as $q$-analogues of multiple zeta values since one can show that for $k_1 \geq 2$
\begin{align}\label{eq:gareqanalogue}
\lim\limits_{q\rightarrow 1}(1-q)^{k_1+\dots+k_r }g(k_1,\ldots,k_r) = \zeta(k_1,\dots,k_r)\,.
\end{align}

In the Fourier expansion of (multiple) Eisenstein series, the $q$-series $g$ always appear together with a power of $-2\pi i$, and therefore we set for $k_1,\dots,k_r \geq 1$
\begin{align*}
     \hat{g}(k_1,\dots,k_r;\tau) =   \hat{g}(k_1,\dots,k_r) := (-2\pi i)^{k_1+\dots + k_r} g(k_1,\dots,k_r) \,.
\end{align*}
With this, a multiple version of $\mathbb{G}(k;\tau) = \zeta(k) + \hat{g}(k)$ is given by the following. 
\begin{theorem}[$r=1,2$ \cite{GKZ}, $r\geq 1$ \cite{B2}]\label{thm:mesfourier}
 For $k_1,\dots,k_r \geq 2$ there exist explicit integers $\alpha^{k_1,\dots,k_r}_{l_1,\dots,l_r,j} \in \Z$, such that for $q=e^{2\pi i \tau}$ we have
\begin{align*}
    \mathbb{G}(k_1,\dots,k_r;\tau) = \zeta(k_1,\dots,k_r) +\!\!\!\!\!\!\!\!\!\!\sum_{\substack{0 < j < r\\l_1+\dots+l_r = k_1+\dots+k_r\\l_1\geq 2,l_2,\dots,l_r\geq 1}} \!\!\!\!\!\!\!\!\!\! \alpha^{k_1,\dots,k_r}_{l_1,\dots,l_r,j}\, \zeta(l_1,\dots,l_j)  \hat{g}(l_{j+1},\dots,l_r) +  \hat{g}(k_1,\dots,k_r)\,.
\end{align*}
In particular, $\mathbb{G}(k_1,\dots,k_r;\tau) = \zeta(k_1,\dots,k_r)+ \sum_{n> 0} a_{k_1,\dots,k_r}(n) q^n$ for $a_{k_1,\dots,k_r}(n) \in \mz[\pi i]$.
\end{theorem}
We will sketch the proof of Theorem \ref{thm:mesfourier} in the following and then give an explicit example at the end of the section. First, observe that for $k_1,\dots,k_r \geq 2$ we have by the Lipschitz formula \eqref{eq:defmonotangent}, that the $q$-series $\hat{g}$ can be written as an ordered sum over monotangent functions
\begin{align}\label{eq:ghatclassical}
    \hat{g}(k_1,\dots,k_r) = \sum_{m_1 > \dots > m_r > 0} \Psi(k_1;m_1 \tau) \cdots \Psi(k_r;m_r \tau) \,.
\end{align}
In general, the multiple Eisenstein series can be written as ordered sums over \emph{multitangent functions} (\cite{Bo}), which are for  $k_1,\dots,k_r \geq 2$ and $\tau \in \Ha$ defined  by
\begin{align}\label{eq:defmultitangent}
    \Psi(k_1,\ldots,k_r ;\tau) := \sum_{\substack{n_1>\cdots >n_r \\n_i \in \Z}} \frac{1}{(\tau+n_1)^{k_1}\cdots (\tau+n_r)^{k_r}}.
\end{align}
These functions were originally introduced by Ecalle and then in detail studied by Bouillot in \cite{Bo}. To write $\mathbb{G}(k_1,\dots,k_r;\tau)$ in terms of these functions, one splits up the summation in the definition \eqref{eq:defmes} into $2^r$ parts, corresponding to the different cases where either $m_i = m_{i+1}$ or $m_i > m_{i+1}$ for $\lambda_i = m_i \tau + n_i$ and $i=1,\dots,{r}$ ($\lambda_{r+1}=0$). Then one can check that the multiple Eisenstein series can be written as 
\begin{align}\label{eq:classicalmesasmouldproduct}
    \mathbb{G}(k_1,\dots,k_r;\tau) = \sum_{j=0}^r \hat{g}^*(k_1,\dots,k_j) \zeta(k_{j+1},\dots,k_r)\,,
\end{align}
where the $q$-series $\hat{g}^*$ are given as ordered sums over multitangent functions by
\begin{align}\label{eq:gastclassical}
    \hat{g}^*(k_1,\dots,k_r) := \sum_{\substack{1 \leq j \leq r\\0 = r_0< r_1 < \dots < r_{j-1} < r_j = r\\ m_1 > \dots > m_j > 0}}  \prod_{i=1}^j \Psi(k_{r_{i-1}+1},\ldots,k_{r_i};m_i \tau)\,.
\end{align}
Further, one can show (\cite[Construction 6.7]{B1}) that the $q$-series $\hat{g}^*$ satisfy the harmonic product formula, e.g. $\hat{g}^*(k_1) \hat{g}^*(k_2) = \hat{g}^*(k_1,k_2) + \hat{g}^*(k_2,k_1) + \hat{g}^*(k_1+k_2)$. We will reformulate this construction in Section \ref{sec:stmes}. To obtain the statement in Theorem \ref{thm:mesfourier}, one then uses the following theorem.

\begin{theorem}{\cite[Theorem 6]{Bo}}\label{thm:reductionmonotangent}
For $k_1,\dots,k_r \geq 2$ with $k=k_1+\dots+k_r$ the multitangent function can be written as 
\begin{align*}
    \Psi(k_1,\dots,k_r;\tau) = \sum_{\substack{1\leq j \leq r\\ l_1+\dots+l_r = k}} (-1)^{l_1+\dots+l_{j-1}+k_j+k} \prod_{\substack{1\leq i \leq r\\i\neq j}}\binom{l_i-1}{k_i-1} \zeta(l_1,\dots,l_{j-1}) \,\Psi_{l_j}(\tau)\, \zeta(l_r,l_{r-1},\dots,l_{j+1})\ .
\end{align*} 
Moreover, the terms with $\Psi(1;\tau)$ vanish. 
\end{theorem}

\begin{proof}
This follows by using partial fraction decomposition 
\begin{align*}
  \frac{1}{(\tau+n_1)^{k_1}\cdots (\tau+n_r)^{k_r}} = \sum_{\substack{1\leq j \leq r\\ l_1+\dots+l_r = k}} \prod_{i=1}^{j-1} \frac{(-1)^{l_i}\binom{l_i-1}{k_i-1}}{(n_i-n_j)^{l_j}} \frac{(-1)^{k+k_j}}{(\tau+n_j)^{l_j}} \prod_{i=l+1}^r \frac{\binom{l_i-1}{k_i-1}}{(n_j-n_i)^{l_j}}\,.
\end{align*}
In order to show that the terms with  $\Psi(1;\tau)$ vanishes, one observes that their coefficient is exactly given by the formula in Proposition \ref{prop:mzvantipoderelation}. And therefore, we will see later, as a simple application of the antipode relation, that these vanish. 

% Define for $N>0$ and any $k_1,\dots,k_r\geq 1$ the truncated version of the multitangent function  
% \begin{align}
%     \Psi^N(k_1,\ldots,k_r ;\tau) &:= \sum_{\substack{N>n_1>\cdots >n_r > -N\\n_i \in \Z}} \frac{1}{(\tau+n_1)^{k_1}\cdots (\tau+n_r)^{k_r}} \\
%     &=\sum_{\substack{1\leq j \leq r\\ l_1+\dots+l_r = k}} \sum_{\substack{N>n_1>\cdots >n_r > -N\\n_i \in \Z}}  \prod_{i=1}^{j-1} \frac{(-1)^{l_i}\binom{l_i-1}{k_i-1}}{(n_i-n_j)^{l_j}} \frac{(-1)^{k+k_j}}{(\tau+n_j)^{l_j}} \prod_{i=l+1}^r \frac{\binom{l_i-1}{k_i-1}}{(n_j-n_i)^{l_j}}\,.
% \end{align}
% Splitting the summation up for a fixed $j$ into the parts $N > n_1 > \dots > n_j$, $N-j+1> n_j > -N+r-j$, $n_j > \dots > n_r > -N$. Setting $\tilde{n}_i = n_i - n_j$  for $i<j$, $\tilde{n}_i = n_j - n_i$ for $i>j$ and $L_j = 2(N-1)-r+j$, $R_j=2N-r-1$ we get a summation over $L_j > \tilde{n}_1 > \dots > \tilde{n}_{j-1} >0$, $N> n_j > -N$, $R_j > \tilde{n}_{r} >\dots > \tilde{n}_{j+1}>0$. We get
% \begin{align}
%    &\Psi^N(k_1,\ldots,k_r ;\tau) := \sum_{\substack{N>n_1>\cdots >n_r > -N\\n_i \in \Z}} \frac{1}{(\tau+n_1)^{k_1}\cdots (\tau+n_r)^{k_r}} \\
%     &=\sum_{\substack{1\leq j \leq r\\ l_1+\dots+l_r = k}} (-1)^{l_1+\dots+l_{j-1}+k_j+k} \prod_{\substack{1\leq i \leq r\\i\neq j}}\binom{l_i-1}{k_i-1} \zeta_{L_j}(l_1,\dots,l_{j-1}) \,\sum_{n=-N+r-j+1}^{N-j} \frac{1}{(\tau+n)^{k_j}}\, \zeta_{R_j}(l_r,l_{r-1},\dots,l_{j+1})\,.
% \end{align}

\end{proof}
Applying Theorem \ref{thm:reductionmonotangent} to \eqref{eq:gastclassical}, we see by \eqref{eq:ghatclassical} that the $\hat{g}^*$ can be written as a $\mz$-linear combination of $\hat{g}$. This proves Theorem \ref{thm:mesfourier} since one can also show that all the appearing multiple zeta values have the correct depth. 

\begin{example}\label{ex:mes32} We give one explicit example in depth two. To write $\mathbb{G}(k_1,k_2)$ as sums over $\Psi$ we consider the following 
\begin{align*}
    \mathbb{G}&(k_1,k_2 ;\tau) = \sum_{m_1 \tau + n_1 \succ m_2 \tau + n_2 \succ 0} \frac{1}{(m_1 \tau + n_1)^{k_1}(m_2 \tau + n_2)^{k_2}} \\
    &= \left( \sum_{\substack{m_1=m_2=0\\n_1>n_2>0}} +   \sum_{\substack{m_1 > m_2=0\\n_1 \in \Z, n_2>0}} + \sum_{\substack{m_1 = m_2>0\\n_1 > n_2}} + \sum_{\substack{m_1 > m_2 > 0\\n_1, n_2 \in \Z}} \right)   \frac{1}{(m_1 \tau + n_1)^{k_1}(m_2 \tau + n_2)^{k_2}} \\
    &= \zeta(k_1,k_2) + \sum_{m>0} \Psi(k_1 ; m \tau) \zeta(k_2) + \sum_{m>0} \Psi(k_1,k_2;m \tau)  + \sum_{m_1>m_2>0} \Psi(k_1;m_1 \tau) \Psi(k_2;m_2\tau) \,.
\end{align*}
These four terms correspond to the following positions of lattice points we are summing over:

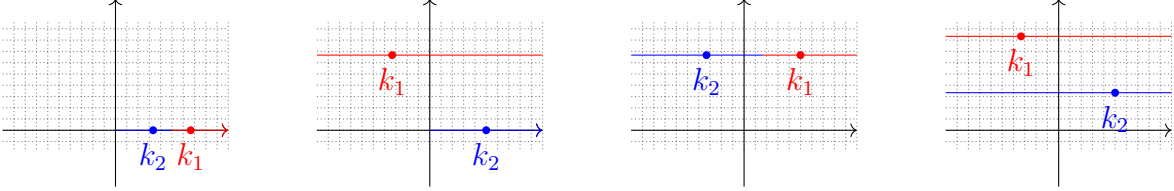
\fboxsep=0pt
\noindent
\begin{minipage}[t]{0.24\linewidth}
	\begin{tikzpicture}[scale=0.5]
		\draw[densely dotted,step=0.3,color=gray,thin] (-3,-0.5) grid (3,2.9);
		\draw [->] (0,-1.5) -- (0,3.5);
		\draw [->] (-3,0) -- (3,0);
		\draw [blue] (0,0) -- (1.5,0);
		\draw [red,->] (1.5,0) -- (3,0);
		\coordinate [label=center:\textcolor{blue}{$k_2$}] (R) at (1,-0.7);
		\coordinate [label=center:\textcolor{red}{$k_1$}] (R) at (2,-0.7);
		\fill[blue] (1,0) circle (3pt);
		\fill[red] (2,0) circle (3pt);
	\end{tikzpicture} 
\end{minipage}
\hfill%
\begin{minipage}[t]{0.24\linewidth}
	\begin{tikzpicture}[scale=0.5]
		\draw[densely dotted,step=0.3,color=gray,thin] (-3,-0.5) grid (3,2.9);
		\draw [->] (0,-1.5) -- (0,3.5);
		\draw [->] (-3,0) -- (3,0);
		\draw [blue,->] (0,0) -- (3,0);
		\draw [red,] (-3,2) -- (3,2);
		\coordinate [label=center:\textcolor{blue}{$k_2$}] (R) at (1.5,-0.7);
		\coordinate [label=center:\textcolor{red}{$k_1$}] (R) at (-1,1.3);
		\fill[blue] (1.5,0) circle (3pt);
		\fill[red] (-1,2) circle (3pt);
	\end{tikzpicture} 
\end{minipage}
\hfill%
\begin{minipage}[t]{0.24\linewidth}
	\begin{tikzpicture}[scale=0.5]
		\draw[densely dotted,step=0.3,color=gray,thin] (-3,-0.5) grid (3,2.9);
		\draw [->] (0,-1.5) -- (0,3.5);
		\draw [->] (-3,0) -- (3,0);
		\draw [red] (0.5,2) -- (3,2);
		\draw [blue] (-3,2) -- (0.5,2);
		\coordinate [label=center:\textcolor{red}{$k_1$}] (R) at (1.5,1.3);
		\coordinate [label=center:\textcolor{blue}{$k_2$}] (R) at (-1,1.3);
		\fill[red] (1.5,2) circle (3pt);
		\fill[blue] (-1,2) circle (3pt);
	\end{tikzpicture} 
\end{minipage}
\hfill%
\begin{minipage}[t]{0.24\linewidth}
	\begin{tikzpicture}[scale=0.5]
		\draw[densely dotted,step=0.3,color=gray,thin] (-3,-0.5) grid (3,2.9);
		\draw [->] (0,-1.5) -- (0,3.5);
		\draw [->] (-3,0) -- (3,0);
		\draw [blue,-] (-3,1) -- (3,1);
		\draw [red,] (-3,2.5) -- (3,2.5);
		\coordinate [label=center:\textcolor{blue}{$k_2$}] (R) at (1.5,0.3);
		\coordinate [label=center:\textcolor{red}{$k_1$}] (R) at (-1,1.8);
		\fill[blue] (1.5,1) circle (3pt);
		\fill[red] (-1,2.5) circle (3pt);
	\end{tikzpicture} 
\end{minipage}
From this we obtain \eqref{eq:classicalmesasmouldproduct}, i.e.
\begin{align*}
    \mathbb{G}_{k_1,k_2}(\tau) = \zeta(k_1,k_2) + \hat{g}^\ast(k_1) \zeta(k_2) + \hat{g}^\ast(k_1,k_2)\,,
\end{align*}
where $\hat{g}^\ast(k_1)=\sum_{m_1>0} \Psi_{k_1}(m_1\tau) = \hat{g}(k_1)$ and 
\begin{align*}
\hat{g}^\ast(k_1,k_2) &=  \sum_{m_1>0} \Psi_{k_1,k_2}(m_1 \tau) + \sum_{m_1>m_2>0}  \Psi_{k_1}(m_1 \tau)\Psi_{k_2}(m_2 \tau)\\
&= \sum_{m_1>0} \Psi_{k_1,k_2}(m_1 \tau) + \hat{g}(k_1,k_2)\,.
\end{align*}
Considering the special case $(k_1,k_2)=(3,2)$ one sees by partial fraction decomposition 
{\small
\begin{align*}
\Psi_{3,2}(\tau) &= \sum_{n_1 > n_2} \frac{1}{(\tau+n_1)^3 (\tau+n_2)^2} \\
&= \sum_{n_1 > n_2}  \left( \frac{1}{(n_1-n_2)^2 (\tau+n_1)^3} +\frac{2}{(n_1-n_2)^3 (\tau+n_1)^2} + \frac{3}{(n_1-n_2)^4 (\tau+n_1)} \right) \\
&+\sum_{n_1 > n_2} \left( \frac{1}{(n_1-n_2)^3 (\tau+n_2)^2}  -  \frac{3}{(n_1-n_2)^4 (\tau+n_2)} \right) =  3 \zeta(3) \Psi_2(\tau) +  \zeta(2) \Psi_3(\tau)  \,,
\end{align*}
}and therefore $\hat{g}^\ast(3,2) =  3 \zeta(3) \hat{g}(2) +  \zeta(2) \hat{g}(3) + \hat{g}(3,2)$. In total, we get 
\begin{align*}
\mathbb{G}(3,2;\tau) =\zeta(3,2) + 3 \zeta(3)  \hat{g}(2) + 2 \zeta(2)  \hat{g}(3) +   \hat{g}(3,2) \,. 
\end{align*}
\end{example}

\section{Algebraic setup}\label{sec:algsetup}

First, we will recall some basic facts on quasi-shuffle products (\cite{H}, \cite{HI}). Let $L$ be a countable set, called \emph{alphabet}, whose elements we will refer to as \emph{letters}. A monic monomial in the non-commutative polynomial ring $\QL$ will be called a \emph{word}, and we denote the empty word by $\one$. 
Suppose we have a commutative and associative product $\diamond$ on the vector space $\Q L$. Then the \emph{quasi-shuffle product}  $\qsh$ on $\QL$ is defined as the $\Q$-bilinear product, which satisfies $\one \qsh w = w \qsh \one = w$ for any word $w\in \QL$ and
\begin{align}\label{eq:qshdef}
	a w \qsh b v = a (w \qsh b v) + b (a w \qsh v) + (a \diamond b) (w \qsh  v) 
\end{align}
for any letters $a,b \in L$ and words $w, v \in \QL$. This gives a commutative $\Q$-algebra $(\QL, \qsh)$, which is called quasi-shuffle algebra.

For describing the algebraic structure for multiple zeta values, we consider two different alphabets. The first is $L_{xy}=\{x,y\}$ together with the product $a \diamond b = 0$ for $a,b \in L_{xy}$. We write $\h = \Q\langle L_{xy}\rangle=\Q\langle x,y \rangle$ and  the corresponding quasi-shuffle product $\sh=\qsh$ is called the \emph{shuffle product}. For example, we have 
\begin{align}\label{eq:shufflexy}
xy \shuffle xxy = xyxxy + 3 xxyxy + 6xxxyy\,.  
\end{align}
We define the following subspaces of $\h$
\begin{align*}
    \h^0 = \Q + x \h y \quad \subset\quad  \h^1 = \Q + \h y \quad \subset \quad \h\,.
\end{align*}
Notice that both spaces $\h^0$ and $\h^1$ are closed under $\sh$ and we denote the corresponding $\Q$-algebras by $\h^0_\sh$ and $\h^1_\sh$.

The second alphabet is $L_z=\{z_k \mid k\geq 1\}$ together with the product $z_{k_1} \diamond z_{k_2} = z_{k_1+k_2}$ for $k_1,k_2\geq 1$. The corresponding quasi-shuffle product $\ast=\qsh$ is called the \emph{stuffle product}. For example, we have
\begin{align}\label{eq:23stuffle}
z_2 \ast z_3 = z_2 z_3 + z_3 z_2 + z_5\,.
\end{align}
By identifying $z_k \leftrightarrow \overbrace{x\cdots x}^{k-1}y$, we can identify $\Q\langle L_z\rangle$ with $\h^1$, and we will not distinguish between them in the following, i.e., we will view $z_k$ as elements in $\h$. The space $\h^1$ equipped with the stuffle product gives a commutative $\Q$-algebra $\h^1_\ast$ with subalgebra $\h^0_\ast$. Moreover, for an index $\kk = (k_1,\dots,k_r) \in \Z_{\geq 1}^r$ we define the word $z_\kk = z_{k_1} \dots z_{k_r}$. Notice that, as a $\Q$-vector space, $\h^1$ is spanned by $z_\kk$ for arbitrary indices $\kk$ and $\h^0$ is spanned by $z_\kk$ for admissible indices $\kk$. 
By above identification, we can write \eqref{eq:shufflexy} as
\begin{align}\label{eq:23shuffle}
z_2 \shuffle z_3 = z_2 z_3 + 3z_3 z_2 + 6z_4z_1\,.
\end{align}
We also consider the following subspace of $\h^0$
\begin{align*}
    \h^{2} = \Q+\langle k_1,\dots,k_r \mid r\geq 1, k_1,\dots, k_r \geq 2\rangle_\Q\,,
\end{align*}
which is spanned by $z_\kk$ such that the multiple Eisenstein series $\mathbb{G}(\kk)$ is defined. Notice that both $\h^0$ and $\h^2$ are closed under $\ast$ but only $\h^0$ is closed under $\shuffle$ as we can see by \eqref{eq:23shuffle}. We obtain the following inclusion of $\Q$-algebras
\begin{align*}
	    \hz_\ast 	\subset    \h^0_\ast &\subset \h^1_\ast \,,\\
    \h^0_\shuffle &\subset \h^1_\shuffle \subset \h_\shuffle \,.
\end{align*}

In this note, we will consider various different objects defined for indices (e.g. multiple Hurwitz zeta functions, multiple zeta values, multiple Eisenstein series, multitangent functions, etc.). In most cases, we want to consider these objects as maps from one of the subspaces of $\h$ into some $\Q$-algebra. By abuse of notation, we will not distinguish between the maps and the objects. 
For example, for any $N\geq 1$, the truncated Hurwitz zeta function can be viewed as a $\Q$-linear map defined on the generators by\footnote{By $\mathcal{O}(U)$ we denote the ring of holomorphic functions on $U \subset \C$. Most of the functions we consider are holomorphic on $U=\C \backslash \Z$, but later we will restrict to the case $U=\Ha$.} 
\begin{align*}
	\zeta_N({-};x) : \h^1 &\longrightarrow \mathcal{O}(\C \backslash \Z),\\
	w = z_{k_1} \dots z_{k_r} &\longmapsto \zeta_N(w;x) := \zeta_N(k_1,\dots,k_r; x)
\end{align*}
and $\zeta_N(\one;x)=1$. In general, for all maps with domain $\h^1$ in this note, we will always assume that the empty word $\one$ gets mapped to $1$.
\begin{lemma}\label{lem:zetanalghom} For $N\geq 1$ the map $\zeta_N({-};x): \h^1_\ast \rightarrow \mathcal{O}(\C \backslash \Z)$ is an algebra homomorphism. 
\end{lemma}
\begin{proof} The proof of this is the same as for (truncated) multiple zeta values and is a special case of \cite[Lemma 2.18]{B3}. 
\end{proof}

Since the limit of $\zeta_N(w;0)$ as $N \rightarrow \infty$ just exists in the case $w \in \h^0$, we define
\begin{align}\label{eq:zetamap}
    \zeta: \h^0 &\longrightarrow \mz\,,\\
   w= z_\kk &\longmapsto \zeta(w)=\zeta(\kk)\,.
\end{align} 
For $\bullet \in \{\ast, \shuffle\}$ the map \eqref{eq:zetamap} gives $\Q$-algebra homomorphism  $\zeta: \h^0_\bullet \longrightarrow \mz$. This is a consequence of the definition as iterated sums (or Lemma \ref{lem:zetanalghom}) and the representation as iterated integrals (see \cite[Section 2]{B0}). For $w,v \in \h^0$ we obtain the relations 
\begin{align}\label{eq:finitedsh}
\zeta(w \ast v)=\zeta(w) \zeta(v)= \zeta(w \sh v)\,,
\end{align}
which are called \emph{finite double shuffle relations}.

\newcommand{\reg}{\operatorname{reg}}
\subsection{Regularization}
In this section, we will recall some results from \cite{IKZ}. For $\bullet \in \{\sh,\ast \}$ any element $w \in \h^1_\bullet$ can be written as a polynomial in $z_1=y$ with coefficients in $\h^0_\bullet$, i.e. there exist some $m\geq 0$ and $c_j(w) \in \h_0$ such that $w =\sum_{j=1}^m c_j(w) \bullet z_1^{\bullet j}$. Since this representation is unique, we obtain algebra isomorphisms
\begin{align*}
    \reg_\bullet: \h^1_\bullet &\longrightarrow \h^0_\bullet[T],\\
    w &\longmapsto   \sum_{j=1}^m c_j(w) T^{j} = \reg_\bullet(w)\,.
\end{align*}
In particular, any algebra homomorphism $f: \h^0_\bullet \rightarrow A$ into some $\Q$-algebra $A$ can be lifted to an algebra homomorphism 
\begin{align}\label{eq:regf}
f^\bullet: \h^1_\bullet &\longrightarrow A[T],\\
w &\longmapsto \sum_{j=1}^m f(c_j(w)) T^{j}\,,
\end{align}
i.e. we set $f^\bullet = f \circ \reg_\bullet$ after extending $f$ to $\h^0_\bullet[T]$ coefficient-wise. In the case of the multiple zeta values map \eqref{eq:zetamap}, this gives the \emph{shuffle regularized multiple zeta values} $\zeta^\sh$ and the \emph{stuffle regularized multiple zeta values} $\zeta^\ast$. By the work, \cite{IKZ}, these two regularizations differ, but their difference can be described explicitly: Define the $\R$-linear map $\rho: \R[T]\rightarrow \R[T]$ by 
\begin{align}\label{eq:defrho}
 \rho(e^{Tu}) := \exp\left(Tu+ \sum_{n=2}^{\infty} \frac{(-1)^n}{n} \zeta(n) u^n \right)	\,.
\end{align}
Then we have $\zeta^\sh = \rho \circ \zeta^\ast$ (\cite[Theorem 1]{IKZ}). The relations among multiple zeta values obtained from this comparison together with the finite double shuffle relations are the \emph{extended double shuffle relations}, which conjecturally give all relations among multiple zeta values. 

 \newcommand{\Hom}{\operatorname{Hom}}
\subsection{Hopf algebra structures}
By the work of Hoffman (\cite{H},\cite{HI}), any quasi-shuffle algebra can be equipped with the structure of a Hopf algebra \cite[Section 3]{H}, where the coproduct is given for $w \in \QL$ by \emph{deconcatenation coproduct}
\begin{align}\label{eq:coproduct}
    \Delta(w) = \sum_{uv = w} u \otimes v\,.
\end{align}
The antipode in this Hopf algebra can be described explicitly (\cite[Theorem 3.2]{H}). For example, if $\diamond$ is the trivial product, then the corresponding antipode for $a_1,\ldots,a_r\in L$ is given by
\begin{align} \label{eq:antipode} 
S(a_1\cdots a_r)=(-1)^r a_r\dots a_1\,.
\end{align}
For any Hopf algebra $A$ with coproduct $\Delta$ and an $\Q$-algebra $B$ with multiplication $m$ and $f,g \in \Hom(A,B)$ the \emph{convolution product} is defined by
\begin{align*}
	f \star g = m \circ (f \otimes g) \circ \Delta\,.
\end{align*}
As a simple fact from the theory of Hopf algebras, we get the following lemma, which will play an important role in the construction in the next section.
\begin{lemma}\label{lem:convprod} If $f,g \in \Hom(A,B)$ then $f \star g \in \Hom(A,B)$.
\end{lemma}

The antipode $S: A \rightarrow A$ is the inverse of $\operatorname{Id}$ with respect to $\star$, i.e. 
\begin{align*}
(S \star \operatorname{Id})(w) = \begin{cases} 1, w=\emptyset\\0, \text{else}\end{cases}\,.
\end{align*}
As a direct consequence of \eqref{eq:antipode} one obtains that for any non-empty word $w=a_1\dots a_m$ in $A=\QL$ we have
\begin{align}\label{eq:antipodewordrelation}
\sum_{i=0}^m (-1)^i a_ia_{i-1}\dots a_1 \shuffle a_{i+1}a_{i+2}\dots a_m=0.
\end{align}
This can be used to prove the following relations among the shuffle regularized multiple zeta values. 
\begin{proposition}\label{prop:mzvantipoderelation}
For $k_1,\dots,k_r \geq 1$ and $k=k_1+\dots+k_r$ we have
\begin{align*}
    \sum_{\substack{1\leq j \leq r\\ l_1+\dots+l_{j-1}+l_{j+1}+\dots+l_r = k-1}} (-1)^{e_j} \prod_{\substack{1\leq i \leq r\\i\neq j}}\binom{l_i-1}{k_i-1} \zeta^\shuffle(l_1,\dots,l_{j-1}) \zeta^\shuffle(l_r,l_{r-1},\dots,l_{j+1}) = 0\,,
\end{align*}
where $e_j = l_1+\dots+l_{j-1}+k_j$. 
\end{proposition}
\begin{proof} By induction and the definition of the shuffle product, one can show that for any $n\geq 1$ 
\begin{align*}
\zeta^\shuffle(x^{k_1-1}y \dots x^{k_r-1}y x^n)= (-1)^n \sum_{l_1+\dots+l_r=k_1+\dots+k_r+n} \prod_{i=1}^r \binom{l_i-1}{k_i-1} \zeta^\shuffle(l_1,\dots,l_r)\,.
\end{align*}
The statement then follows by using the following relation, which is a consequence of \eqref{eq:antipodewordrelation} for $a_m\dots a_1 = x^{k_1-1}y \dots x^{k_r-1}$:
\begin{align*}
	 \sum_{i=0}^m (-1)^i \zeta^\shuffle(a_1 \dots a_i)\zeta^\shuffle( a_m a_{m-1} \dots a_{i+1}) =0\,.
\end{align*}
\end{proof}

\section{Stuffle regularized multiple Eisenstein series}\label{sec:stmes}
Using the algebraic setup described in the previous section, we will now give a new interpretation of the regularization of the multitangent functions in \cite[Section 7]{Bo} and the stuffle regularization of multiple Eisenstein series presented in \cite[Section 6]{B1}. The following construction is motivated by the original calculation of the Fourier expansion explained in Section \ref{sec:mes}. To calculate the Fourier expansion for the double Eisenstein series $\mathbb{G}(k_1,k_2)$ one considers for $k_1,k_2 \geq 2$ the four terms
{\small
\begin{align*}
    \mathbb{G}(k_1,k_2 ;\tau)= \zeta(k_1,k_2) + \sum_{m>0} \Psi(k_1 ; m \tau) \zeta(k_2) + \sum_{m>0} \Psi(k_1,k_2;m \tau)  + \sum_{m_1>m_2>0} \Psi(k_1;m_1 \tau) \Psi(k_2;m_2\tau) \,.
\end{align*}}
Combining some of these terms into 
\begin{align}\begin{split}\label{eq:gastinpsi12}
\hat{g}^\ast(k_1) &=  \sum_{m>0} \Psi(k_1 ; m \tau),\\
\hat{g}^\ast(k_1,k_2) &= \sum_{m>0} \Psi(k_1,k_2;m \tau)  + \sum_{m_1>m_2>0} \Psi(k_1;m_1 \tau) \Psi(k_2;m_2\tau)\,
\end{split}
\end{align}
we obtain $
\mathbb{G}(k_1,k_2 ;\tau)= \zeta(k_1,k_2) +  g^\ast(k_1) \zeta(k_2) +g^\ast(k_1,k_2)$.

\subsection{Construction of stuffle regularized multiple Eisenstein series}
We now want to generalize this idea to the truncated versions. For this, we define for $M\geq 1$ 
\begin{align*}
    \Z_M = \{ m \in \Z \mid |m|<M\}\,.
\end{align*}
and for $\tau \in \Ha$ define on $\Z \tau + \Z$ the order $\succ$ as before by 
\begin{align*}
    m_1 \tau + n_1 \succ m_2 \tau + n_2 \quad :\Leftrightarrow \quad (m_1 > m_2) \,\text{  or  }\, (m_1 = m_2  \text{ and } n_1 > n_2 )\,.
\end{align*}
We illustrate this order in the following diagram:
\vspace{-0.5cm}
\begin{figure}[H]
    \begin{center}
        \begin{tikzpicture}[scale=0.4]
\draw[dotted,step=1,color=gray,thin] (-6.9,-2) grid (6.9,5.9); %densely dotted,
\draw [->,thick] (0,-1) -- (0,6.5) node (yaxis) [above] {$m$};
\draw [->,thick] (-7,0) -- (7,0) node (xaxis) [right] {$n$};

\draw[black] (-1,-0.2) -- (-1,0.2);
\draw[black] (-2,-0.2) -- (-2,0.2);
\draw[black] (-3,-0.2) -- (-3,0.2);
\draw[black] (-4,-0.2) -- (-4,0.2);
\draw[black] (-5,-0.2) -- (-5,0.2);
\draw[red,thick] (-6,-0.2) -- (-6,0.2);

\draw[black] (1,-0.2) -- (1,0.2);
\draw[black] (2,-0.2) -- (2,0.2);
\draw[black] (3,-0.2) -- (3,0.2);
\draw[black] (4,-0.2) -- (4,0.2);
\draw[black] (5,-0.2) -- (5,0.2);
\draw[red,thick] (6,-0.2) -- (6,0.2);

\draw[blue,thick] (-0.2,5) -- (0.2,5);

\coordinate [label=center:\textcolor{red}{$-6$}] (R) at (-6,-0.7);
\coordinate [label=center:\textcolor{red}{$6$}] (R) at (6,-0.7);

\coordinate [label=center:\textcolor{blue}{$5$}] (R) at (-0.6, 5);

\fill[orange] (-5,1) circle (4pt);
\fill[orange] (-4,1) circle (4pt);
\fill[orange] (-3,1) circle (4pt);
\fill[orange] (-2,1) circle (4pt);
\fill[orange] (-1,1) circle (4pt);
\fill[orange] (0,1) circle (4pt);
\fill[orange] (1,1) circle (4pt);
\fill[orange] (2,1) circle (4pt);
\fill[orange] (3,1) circle (4pt);
\fill[orange] (4,1) circle (4pt);
\fill[orange] (5,1) circle (4pt);
\fill[orange] (-5,2) circle (4pt);
\fill[orange] (-4,2) circle (4pt);
\fill[orange] (-3,2) circle (4pt);
\fill[orange] (-2,2) circle (4pt);
\fill[orange] (-1,2) circle (4pt);
\fill[orange] (0,2) circle (4pt);
\fill[orange] (1,2) circle (4pt);
\fill[orange] (2,2) circle (4pt);
\fill[orange] (3,2) circle (4pt);
\fill[orange] (4,2) circle (4pt);
\fill[orange] (5,2) circle (4pt);
\fill[orange] (-5,3) circle (4pt);
\fill[orange] (-4,3) circle (4pt);
\fill[orange] (-3,3) circle (4pt);
\fill[orange] (-2,3) circle (4pt);
\fill[orange] (-1,3) circle (4pt);
\fill[orange] (0,3) circle (4pt);
\fill[orange] (1,3) circle (4pt);
\fill[orange] (2,3) circle (4pt);
\fill[orange] (3,3) circle (4pt);
\fill[orange] (4,3) circle (4pt);
\fill[orange] (5,3) circle (4pt);
\fill[orange] (-5,4) circle (4pt);
\fill[orange] (-4,4) circle (4pt);
\fill[orange] (-3,4) circle (4pt);
\fill[orange] (-2,4) circle (4pt);
\fill[orange] (-1,4) circle (4pt);
\fill[orange] (0,4) circle (4pt);
\fill[orange] (1,4) circle (4pt);
\fill[orange] (2,4) circle (4pt);
\fill[orange] (3,4) circle (4pt);
\fill[orange] (4,4) circle (4pt);
\fill[orange] (5,4) circle (4pt);
%\fill[orange] (-5,5) circle (4pt);
%\fill[orange] (-4,5) circle (4pt);
%\fill[orange] (-3,5) circle (4pt);
%\fill[orange] (-2,5) circle (4pt);
%\fill[orange] (-1,5) circle (4pt);
%\fill[orange] (0,5) circle (4pt);
%\fill[orange] (1,5) circle (4pt);
%\fill[orange] (2,5) circle (4pt);
%\fill[orange] (3,5) circle (4pt);
%\fill[orange] (4,5) circle (4pt);
%\fill[orange] (5,5) circle (4pt);
\fill[orange] (1,0) circle (4pt);
\fill[orange] (2,0) circle (4pt);
\fill[orange] (3,0) circle (4pt);
\fill[orange] (4,0) circle (4pt);
\fill[orange] (5,0) circle (4pt);

%\draw[red,very thin] (1,0.5) arc (90:270:0.5);
%\draw[red,very thin] (1,0.5) -- (7,0.5);
%\draw[red,very thin] (1,-0.5) -- (7,-0.5);
%\coordinate [label=right:\textcolor{red}{$R$}] (R) at (6.5,0);
%\draw[blue,very thin] (-6,0.7) -- (7,0.7);
%\coordinate [label=right:\textcolor{blue}{$U$}] (R) at (6.5,3);
\end{tikzpicture}
        
        All the points $
            {\color{orange}\lambda}\in \Z_{{\color{blue} 5}} i + \Z_{{\color{red} 6}}$ satisfying ${\color{orange}\lambda \succ 0}$.
    \end{center}
\end{figure}

\begin{definition}
For integers $k_1,\ldots,k_{r}\ge1$, and $M,N\ge 1$ we define the \emph{ truncated multiple Eisenstein series}  by 
\[ \mathbb{G}_{M,N}(k_1,\ldots,k_r ;\tau) =   \sum_{\substack{ \lambda_1\succ \cdots\succ \lambda_r \succ 0\\ \lambda_i\in\Z_M \tau+\Z_N}} \frac{1}{\lambda_1^{k_1}\cdots \lambda_r^{k_r}} \,. \]
\end{definition}
For $k_1,\dots,k_r\geq 2$, we obtain the multiple Eisenstein series from Section \ref{sec:mes} by\footnote{Notice that we also allow $k_1=2$ here. In this case, the sum is not absolutely convergent but conditionally convergent and all results mentioned in Section \ref{sec:mes} still hold.}
\begin{align*}
   \mathbb{G}(k_1,\ldots,k_r ;\tau) = \lim_{M \rightarrow \infty}  \lim_{N \rightarrow \infty}  \mathbb{G}_{M,N}(k_1,\ldots,k_r ;\tau)\,.
\end{align*}

Fixing $M,N\geq 1$ we can view the truncated multiple Eisenstein series as $\Q$-linear maps $\mathbb{G}_{M,N}(-;\tau): \h^1 \rightarrow \mathcal{O}(\Ha)$. One can check directly that these are algebra homomorphism with respect to the stuffle product $\ast$. We will show this in the following by rewriting them as the convolution product of two other algebra homomorphism. For this we first define the following.
\begin{definition}
For $k_1,\dots,k_r \geq 1$,$N\geq 1$ and $x \in \C \backslash \Z$ define the  \emph{truncated multitangent function} by
\[ \Psi_N(k_1,\ldots,k_r; x) := \sum_{\substack{N>n_1>\cdots >n_r > -N\\n_i \in \Z}} \frac{1}{(x+n_1)^{k_1}\cdots (x+n_r)^{k_r}}.\]
\end{definition}
For $k_1,k_r \geq 2$ the {multitangent function from Section \ref{sec:mes} are given by $\Psi(k_1,\ldots,k_r; x) = \lim_{N\rightarrow \infty} \Psi_N(k_1,\ldots,k_r; x)\,.$

\begin{definition}
For $M,N\geq 1$ define the map $\hat{g}_{M,N}({-};\tau): \h^1 \rightarrow \mathcal{O}(\Ha)$  by 
\begin{align*}
	\hat{g}_{M,N}({-};\tau) = \cop_{m=1}^{M-1} \Psi_N({-}; m\tau) \,,
\end{align*}
where we write $\cop_{j=a}^{b} f_j= f_{b} \star  f_{b-1} \star \dots \star f_a$.
\end{definition}
Notice that this generalizes \eqref{eq:gastinpsi12} for the truncated version, since for $w\in \h^1$ we have 
\begin{align*}
    \hat{g}_{M,N}(w;\tau) = \sum_{\substack{j \geq 1\\ w_1 \dots w_j = w \\ w_1,\dots,w_j \not = \emptyset}} \sum_{M> m_1 > \dots > m_j > 0} \Psi_N(w_1; m_1 \tau ) \dots \Psi_N(w_j; m_j \tau)\,.
\end{align*}
As a truncated version of  \eqref{eq:classicalmesasmouldproduct} we get the following.
\begin{proposition}
For any $M,N\geq 1$ we have 
\begin{align*}
	\mathbb{G}_{M,N}  = \hat{g}_{M,N} \star \zeta_N\,.
\end{align*}
\end{proposition}
\begin{proof} The argument for this is the same as for the usual multiple Eisenstein series as described in \cite{B0}, \cite{B3} or \cite{BT}. For a given word $w=z_{k_1}\cdots z_{k_r}$ the terms in $\mathbb{G}_{M,N}(w)$ are grouped into $2^r$ groups, which all can be expressed as products of $\hat{g}_{M,N}(w_2)$ and $\zeta_N(w_1)$ with $w=w_1 w_2$.
\end{proof}
Define for $k_1,\dots,k_r \geq 1$, $x\in \C \backslash \Z$ and $N\geq 1$
\begin{align*}
    C(k_1,\dots,k_r;x) &= \begin{cases} 1, &r=0,\\
        \frac{1}{x^{k_1}}, &r=1,\\
        0, &r\geq 2
    \end{cases}\,,\\
    \zeta^{-}_N(k_1,\dots,k_r; x) &= \sum_{0> n_1 > \dots > n_r > -N} \frac{1}{(x+n_1)^{k_1} \dots (x + n_r)^{k_r}}\,.
\end{align*}
Again we can view these as $\Q$-linear maps $C, \zeta^{-}_N: \h^1 \rightarrow \mathcal{O}(\C \backslash \Z)$ defined on $z_{k_1}\dots z_{k_r}$ by the above formulas. 
\begin{proposition}\label{prop:czetanalghom}
The maps $C, \zeta^{-}_N: \h^1_\ast \rightarrow \mathcal{O}(\C \backslash \Z)$ are algebra homomorphisms. 
\end{proposition}
\begin{proof} For the map $C$ this is obvious, since $C(k_1;x) C(k_2;x) = C(k_1+k_2;x)$ and the higher depth products are trivial. For $\zeta^{-}_N$ observe that $\zeta^{-}_N(k_1,\dots,k_r;x) = (-1)^{k_1+\dots+k_r} \zeta_N(k_r,\dots,k_1;-x)$. The statement now follows from the well-known fact that the $\Q$-linear maps defined on the generators by $z_{k_1}\cdots z_{k_r} \mapsto z_{k_r} \cdots z_{k_1}$ and $z_{k_1}\cdots z_{k_r} \mapsto (-1)^{k_1+\dots+k_r} z_{k_1}\cdots z_{k_r} $ are algebra homormorphism on $\h^1_\ast$ together with Lemma \ref{lem:zetanalghom}.
\end{proof}

\begin{proposition}
For $N\geq 1$ we have 
\begin{align}\label{eq:truncatedmultitangent}
    \Psi_N({-};x) = \zeta_N({-};x) \star C({-};x) \star \zeta^{-}_N({-};x)\,.
\end{align}
\end{proposition}
\begin{proof} This follows immediately from the definition.
\end{proof}

The limit $N\rightarrow \infty$ of \eqref{eq:truncatedmultitangent} evaluated at a word which is not in $\hz$ does, in general, not exist. To overcome this problem, we use that the multiple Hurwitz zeta function can be regularized (c.f. \cite{Bo}, \cite{KXY}) to algebra homomorphism $\zeta^\ast( {-}; x): \h^1_\ast \rightarrow \mathcal{O}(\Ha)$, such that
\begin{enumerate}[(i)]
\item For $k_1 \geq 2$ we have $\zeta^\ast(k_1,\dots,k_r;x) = \lim_{N\rightarrow \infty} \zeta_N(k_1,\dots,k_r;x)$\,,
\item $\zeta^\ast(1;x) = \sum_{n>0} \left( \frac{1}{n+x} - \frac{1}{n} \right)$\,.
\end{enumerate}
This follows again from the regularization \eqref{eq:regf}, since for $w\in \h^0$ and $x \in \Ha$ the limit $\lim_{N\rightarrow \infty} \zeta_N(w;x)$ exists. This gives an algebra homomorphism $\zeta({-}; x): \h^0_\ast \rightarrow \mathcal{O}(\Ha)$ which, by \eqref{eq:regf}, can be lifted to an algebra homomorphism from $\h^1_\ast$ to $\mathcal{O}(\Ha)[T]$. By sending $T$ to $\sum_{n>0} \left( \frac{1}{n+x} - \frac{1}{n} \right)$ we then obtain the algebra homomorphism $\zeta^\ast( {-}; x): \h^1_\ast \rightarrow \mathcal{O}(\Ha)$. For $k_1,\dots,k_r\geq 1$, we then define $\zeta^{-}_N(k_1,\dots,k_r;x) = (-1)^{k_1+\dots+k_r} \zeta_N(k_r,\dots,k_1;-x)$, which also gives (with the same argument as in the proof of Proposition \ref{prop:czetanalghom}) an algebra homomorphism $\zeta^\ast( {-}; x): \h^1_\ast \rightarrow \mathcal{O}(\Ha)$. Using these two we give the following definition of stuffle regularized multitangent functions. 

\begin{definition} \label{def:regmultitangent} \begin{enumerate}[(i)]
\item 
We define the algebra homomorphism $\Psi^\ast: \h^1_\ast \rightarrow \mathcal{O}(\Ha)$ by 
\begin{align*}
    \Psi^\ast({-};\tau) = \zeta^\ast({-};\tau) \star C({-};\tau) \star \zeta^{-,\ast}({-};\tau)\,.
\end{align*}
\item  For $M\geq 1$ define  the algebra homomorphism $\hat{g}^\ast_M: \h^1_\ast \rightarrow \mathcal{O}(\Ha)$
\begin{align*}
\hat{g}^\ast_{M}({-};\tau) &= \cop_{m=1}^{M-1} \Psi^\ast({-}; m\tau) \\
&= \cop_{m=1}^{M-1} \Big( \zeta^\ast({-};m\tau) \star C({-};m\tau) \star \zeta^{-,\ast}({-};m\tau) \Big)\,.
\end{align*}
\end{enumerate}
\end{definition}

The fact that both maps $\Psi^\ast$ and $\hat{g}^\ast_{M}({-};\tau)$ are algebra homormorphism again follow directly by Lemma \ref{lem:convprod}, \ref{lem:zetanalghom} and Proposition \ref{prop:czetanalghom}. The stuffle regularized multitangent functions $\Psi^\ast(k_1,\dots,k_r;x)$ in Definition \ref{def:regmultitangent}  coincide with the "symmetrel extension of multitangent functions" constructed in \cite{Bo}. In \cite{Bo}, this is done by using the language of moulds and mould product instead of using the convolution product in the Hopf algebra $\h^1_\ast$.

\begin{proposition} For all $w \in \h^0$ the limit $\hat{g}^\ast(w;\tau) := \lim_{M \rightarrow \infty}	\hat{g}^\ast_{M}(w;\tau)$ exists. In particular, this gives a $\Q$-algebra homomorphism $\hat{g}^\ast(-;\tau): \h^0 \rightarrow \mathcal{O}(\Ha)$.
\end{proposition}
\begin{proof} First notice that $\Psi^\ast({-};\tau)$ is, by construction, $1$-periodic in $\tau$ and therefore also possesses an expansion in $q$. The results follow from the fact that one can check that $\Psi^\ast(1,\dots,1;\tau)$ is the only regularized multitangent function with a non-vanishing constant term in its $q$-expansion. From this, one can then show that for each $n\geq 1$, there exists some $M(n)$, such that the coefficient of $q^n$ in $\hat{g}^\ast_{M}(w;\tau)$ with $w\in \h^0$ is the same for all $M>M(n)$. The details for this will be worked out in \cite{T}. 
\end{proof}

Again we can use the regularization \ref{eq:regf} to define an algebra homomorphism $\hat{g}^\ast: \h^1_\ast \rightarrow \mathcal{O}(\Ha)$ which satisfies the following
\begin{enumerate}[(i)]
\item $\hat{g}^\ast(w;\tau) = \lim_{M \rightarrow \infty} 	\hat{g}_{M}(w;\tau)$ for $w \in \h^0$.
\item $\hat{g}^\ast(z_1;\tau) = (-2\pi i) \sum_{m,n\geq 1} q^{mn}$.
\end{enumerate}
Notice that $\hat{g}^\ast(z_1;\tau)$ coincide with the $\hat{g}(1;\tau)$ in \eqref{eq:defmonog} and that for $k_1,\dots,k_r\geq 2$ we have $\hat{g}^\ast(k_1,\dots,k_r;\tau) = \hat{g}(k_1,\dots,k_r;\tau)$. 

\begin{definition} Define the \emph{stuffle regularized multiple Eisenstein series} as the algebra homomorphism $\mathbb{G}^\ast: \h^1_\ast \rightarrow \mathcal{O}(\Ha)$ 
	\vspace{-0.3cm}
	\begin{align*}
		\mathbb{G}^\ast = \hat{g}^\ast \star \zeta^\ast\,.
	\end{align*}
\end{definition}

\begin{theorem}
We have $\mathbb{G}^\ast_{\mid \hz} = \mathbb{G}$.
\end{theorem}
\begin{proof} The construction of $\mathbb{G}^\ast$ followed exactly the calculation of the Fourier expansion of $\mathbb{G}$. The statement follows by checking that for indices with entries $\geq 2$ the $\hat{g}^\ast$ are exactly the same as the $\hat{g}$ in Section \ref{sec:mes}.
\end{proof}

\subsection{Comparison to shuffle regularized multiple Eisenstein series}
In \cite{G}, Goncharov introduces the Hopf algebra of formal iterated integrals. The coproduct in this Hopf algebra has an explicit combinatorial description. In \cite{BT} it was shown that, after dividing out a certain ideal, one obtains a Hopf algebra which, as an algebra, is isomorphic to $\h^1_\shuffle$. Therefore we can equip $\h^1_\shuffle$ with the \emph{Goncharov coproduct}, which we denote by $\Delta_G$. By the explicit formulas for $\Delta_G$ we obtain, for example
\begin{align*}
    \Delta_G(z_3 z_2) &= z_3 z_2 \otimes  1 + 3 z_2 \otimes z_3 + 2 z_3 \otimes z_2 + 1 \otimes z_3 z_2 \,.
\end{align*}
Compare this to the Fourier expansion of $\mathbb{G}(3,2)$: 
{\small
    \begin{align*}
        \mathbb{G}(3,2 ; \tau) &=\zeta(3,2) + 3   \hat{g}(2) \zeta(3) + 2  \hat{g}(3) \zeta(2)  +  \hat{g}(3,2)\,.
\end{align*}}
In \cite{BT} it was shown that the coproduct can always be used to describe the Fourier expansion of multiple Eisenstein series. More precisely, setting $f \star_G g = m \circ (f \otimes g) \circ \Delta_G$ it was shown the following:
\begin{theorem}[\cite{BT}]
We have $\mathbb{G} = (\hat{g} \star_G \zeta)_{\mid \hz}$.
\end{theorem}

\begin{proposition}{\cite{BT}}
There exists an algebra homomorphism $\hat{g}^\shuffle: \h^1_\shuffle \rightarrow \mathcal{O}(\Ha)$ with $\hat{g}^\shuffle_{\mid \hz} = \hat{g}$.
\end{proposition}

\begin{definition}
Define the \emph{shuffle regularized multiple Eisenstein series} as the algebra homomorphism $\mathbb{G}^\shuffle: \h^1_\shuffle \rightarrow \mathcal{O}(\Ha)$ 
\begin{align*}
	\mathbb{G}^\shuffle = \hat{g}^\shuffle \star \zeta^\shuffle\,.
\end{align*}
\end{definition}

By the previously mentioned results we have
\begin{align*}
	\mathbb{G}^\shuffle_{\mid \hz}	= \mathbb{G} = \mathbb{G}^\ast_{\mid \hz}\,,
\end{align*}
which has the following result as a consequence. 

\begin{corollary}
The shuffle regularized multiple Eisenstein series satisfy the \emph{restricted double shuffle relations}, i.e. for $w,v \in \hz$
\begin{align*}
	\mathbb{G}^\sh( w \shuffle v - w \ast v) = 0\,.
\end{align*}
\end{corollary}

As it was already observed numerically at the end of \cite{BT}, there seem to be more relations satisfied by the $\mathbb{G}^\sh$ than just the restricted double shuffle relations. By comparing $\mathbb{G}^\ast$ and $\mathbb{G}^\sh$ explicitly in depth three, one can obtain the following result, which will be one of the main results of \cite{T}.

\begin{proposition}[\cite{T}]For $k_1, k_3\geq 2, k_2 \geq 1 $ the shuffle regularized multiple Eisenstein series satisfy the finite double shuffle relations
\begin{equation*}
    \mathbb{G}^\sh(z_{k_1}z_{k_2}\sh z_{k_3}-z_{k_1}z_{k_2}\ast z_{k_3})=0 \, .
\end{equation*}
\end{proposition}

But in higher depths, it seems that not all finite double shuffle relations are satisfied by the $\mathbb{G}^\sh$, and there is no explicit conjecture yet which relations are satisfied. For this, it might be necessary to understand the differences between $\mathbb{G}^\sh$ and $\mathbb{G}^\ast$ in more detail. In particular, it might be interesting to check if an analogue of the map $\rho$ in \eqref{eq:defrho} exists for the regularizations of multiple Eisenstein series.

\end{document}